\documentclass[10pt]{amsart}

\usepackage{enumerate}

\newtheorem{lem}{Lemma}
\newtheorem{cor}[lem]{Corollary}
\newtheorem{thm}[lem]{Theorem}

\theoremstyle{remark}

\newcommand{\cA}{\mathcal{A}}
\newcommand{\cI}{\mathcal{I}}
\newcommand{\cJ}{\mathcal{J}}
\newcommand{\cM}{\mathcal{M}}
\newcommand{\e}{\varepsilon}

\newcommand{\C}{\mathbb{C}}

\newcommand{\U}{\mathbf{1}}
\newcommand{\mis}{\mathfrak{M}}
\newcommand{\supp}{\mathrm{supp}}

\newcommand{\set}[1]{\{#1\}}
\newcommand{\enorm}{\lVert\,\cdot\,\rVert}
\newcommand{\Bignorm}[1]{\Bigl\lVert #1 \Bigr\rVert}

\begin{document}

\title[maximal ideals in algebras of continuous functions]%
{New characterizations of maximal ideals in algebras of continuous vector-valued functions}%

\author{Mortaza Abtahi}%

\address{School of Mathematics and Computer Sciences, Damghan University,
P.O.Box 36715-364, Damghan, Iran}%

\email{abtahi@du.ac.ir}

\subjclass[2010]{Primary 46J10; secondary 46J20}
\keywords{Continuous vector-valued functions; Commutative Banach algebras; Maximal ideals}%


\begin{abstract}
  Let $X$ be a compact Hausdorf space, let
  $A$ be a commutative unital Banach algebra, and let
  $\cA$ denote the algebra of continuous
  $A$-valued functions on $X$ equipped with the uniform norm
  $\|f\|_u=\sup_{s\in S}\|f(x)\|$ for all $f\in \cA$.
  Hausner, in [Proc. Amer. Math. Soc. \textbf{8}(1957), 246--249],
  proved that $\cM$ is a maximal ideal in $\cA$ if and only if there exist
  a point $x_0\in X$ and a maximal ideal $M$ in $A$ such that
  $\cM=\set{f\in \cA:f(x_0)\in M}$. In this note, we give
  new characterizations of maximal ideals in $\cA$. We also
  present a short proof of Hausner's result by a different approach.
\end{abstract}

\maketitle

\section{Introduction}

Throughout the paper, $X$ is a compact Hausdorff space and
$C(X)$ denotes the algebra of all continuous complex-valued functions on $X$.
We let $(A,\enorm)$ be a complex commutative Banach algebra with identity. We
denote the unit element of $A$ by $\U$ and assume that $\|\U\|=1$.
We denote the algebra of all continuous functions defined on $X$
with values in $A$ by $C(X,A)$. For simplicity, we write $\cA$ to
denote $C(X,A)$. For $f\in \cA$, the \emph{uniform norm} of $f$
is defined by
\[
 \|f\|_u=\sup\set{\|f(x)\|:x\in X}.
\]

\noindent
In this setting, $(\cA,\enorm_u)$ is a commutative Banach algebra.
Given an element $a\in A$, we consider $a$ as a constant function on $X$, and
$A$ is regarded as a closed subalgebra of $\cA$. Also, by identifying every
scalar-valued function $f\in C(X)$ with the vector-valued function $x\mapsto f(x)\U$,
we regard $C(X)$ as a closed subalgebra of $\cA$.

It is well-known that an ideal $M$ in $C(X)$ is maximal if and only if
there exists a point $x_0\in X$ such that $M=\set{f\in C(X): f(x_0)=0}$.
In this note, we investigate the maximal ideals in $C(X,A)$ and give
new characterizations of these maximal ideals.

In 1950, Bertman Yood \cite{Yood} investigated the structure
of the space $\mis(\cA)$ of maximal ideals of $\cA$.
He proved \cite[Theorem 3.1]{Yood} that when $X$ is completely regular and $A$ is a
$B^*$-algebra, $\mis(\cA)$ is homeomorphic to the $\beta$-compactification of
$X\times\mis(A)$. Then, in section 5 of his paper, he assumed
that $X$ is compact but the requirement that $A$ be a $B^*$-algebra
was dropped; he proved that $\mis(\cA)$ is homeomorphic to $X$,
provided $X$ is compact and $A$ is a primary ring; \cite[Theorem 5.2]{Yood}.
He also proved that if every maximal ideal in $\cA$ is of the form
\begin{equation}\label{sp-form}
\cM=\set{f\in \cA : f(x)\in M},
\end{equation}

\noindent
for some $x\in X$ and $M$ a maximal ideal in $A$,
then $\mis(\cA)=X\times \mis(A)$.
A few years later, Alvin Hausner \cite{Hausner} improved the result
of Yood by showing that, when $X$ is compact, every maximal ideal in
$\cA$ is of the form \eqref{sp-form}.
In this paper, by a different approach,
we present a short proof for this result.

\section{Main Results}

For the Banach algebra $A$, a linear functional $\phi:A\to \C$
is called a \emph{character} if $\phi\neq0$ and
$\phi(ab)=\phi(a)\phi(b)$, for all $a,b\in A$.
We denote the set of all multiplicative functionals
on $A$ by $\mis(A)$. By \cite[Theorem 11.5]{Rudin-FA},
there is a one-to-one  correspondence between $\mis(A)$ and
the set of all maximal ideals in $A$.

Ideals in $A$ will be denoted by $I$, $J$, $M$, etc, and ideals
in $\cA$ will be denoted by $\cI$, $\cJ$, $\cM$, etc.
For an ideal $\cI$ in $\cA$, a point $x\in X$ and a
character $\phi\in \mis(A)$, we define
\begin{equation}\label{I(x) and vp(I)}
  \cI(x) = \set{f(x):f\in \cI}, \qquad
  \phi(\cI) = \set{\phi\circ f: f\in \cI}.
\end{equation}

It is easy to check that $\cI(x)$ and $\phi(\cI)$ are ideals in $A$ and
$C(X)$, respectively, and, in case $\cI$ is closed,
both $\cI(x)$ and $\phi(\cI)$ are closed. Moreover,
\[
  \cI\cap A\subset\cI(x), \quad
  \cI\cap C(X) \subset \phi(\cI).
\]

\begin{lem}
  Let $\cI$ be a closed ideal in $\cA$, and let $f\in \cA$. If
  $f(x)\in\cI(x)$, for every $x\in X$, then $f\in\cI$.
  In particular, if $f(X)\subset\cI$ then $f\in \cI$.
\end{lem}

\begin{proof}
  Let $\e>0$. For every $x\in X$, by the assumption, there
  exists a function $g_x\in\cI$ such that $f(x)=g_x(x)$. Therefore,
  for every $x$, there is a neighborhood $V_x$ of $x$
  such that $\|f(t)-g_x(t)\|<\e$, for all $t\in V_x$. Since
  $X$ is compact, there are finitely many points
  $x_1,\dotsc,x_n$ in $X$ such that
  \[
    X\subset V_1\cup \dotsb \cup V_n,
    \qquad (V_i=V_{x_i}).
  \]

  \noindent
  By \cite[Theorem 2.13]{Rudin-RCA}, there are nonnegative continuous functions
  $h_1,\dotsc,h_n$ such that $\supp(h_i)\subset V_i$
  and $h_1+\dotsb+h_n=1$ on $X$. Take
  $g_i=g_{x_i}$ and $g=h_1g_1+\dotsb+h_ng_n$. Then $g\in \cI$
  and, for every $t\in X$,
  \[
    \|f(t)-g(t)\| =
    \Bignorm{\sum_{i=1}^n h_i(t)f(t)-\sum_{i=1}^n h_i(t)g_i(t)} \leq
    \sum_{i=1}^n h_i(t) \|f(t)-g_i(t)\| < \e.
  \]
  This implies that $\|f-g\|_u \leq \e$ and, since $\cI$ is
  closed, we have $f\in\cI$.
\end{proof}

\begin{cor}\label{if I(x)=J(x) for all x then I=J}
  Let $\cI$ and $\cJ$ be closed ideals in $\cA$. If
  $\cI(x)=\cJ(x)$ for every $x$, then $\cI=\cJ$.
  In particular, if $\cI(x)=A$ for every $x$, then
  $\cI=\cA$.
\end{cor}

Let us now give a very short proof of the fact that every
maximal ideal in $\cA$ is of the form \eqref{sp-form}.

\begin{thm}\label{thm:cM is maximal iff}
  Let $\cM$ be a maximal ideal in $\cA$. Then there exist a point
  $x_0\in X$ and a maximal ideal $M$ in $A$ such that
  $\cM=\set{f\in \cA: f(x_0)\in M}$.
\end{thm}

\begin{proof}
  There exists, by Corollary \ref{if I(x)=J(x) for all x then I=J}, a point $x_0\in X$
  such that $\cM(x_0)\neq A$. We prove that the ideal $\cM(x_0)$
  is maximal in $A$. If $a\in A$ and $a\notin\cM(x_0)$, then $a$, as
  a constant function on $X$, does not belong to
  $\cM$ and so, for some $g\in\cA$ and $f\in\cM$, we have
  $\U=ag+f$. In particular, $\U=ag(x_0)+f(x_0)$ meaning
  $A$ is generated by $\cM(x_0)\cup\set{a}$ and hence
  $\cM(x_0)$ is maximal in $A$.
  Now, let $\cM_0=\set{f\in \cA: f(x_0)\in \cM(x_0)}$.
  Then $\cM_0$ is a proper ideal in $\cA$ and
  $\cM\subset \cM_0$. Hence $\cM=\cM_0$.
\end{proof}

\begin{thm}\label{M(x)=M cap A}
  If $\cM$ is a maximal ideal in $\cA$, then $\cM(x_0)$, for some $x_0\in X$,
  is a maximal ideal in $A$, and $\cM(x_1)=A$ for $x_1\neq x_0$. Moreover,
  $\cM(x_0)=\cM\cap A$.
\end{thm}

\begin{proof}
  The proof of Theorem \ref{thm:cM is maximal iff} shows
  that there is a point $x_0\in X$
  such that $\cM(x_0)$ is a maximal ideal in $A$.
  Clearly, $\cM\cap A\subset \cM(x_0)$. Assume that $a\in A$
  and $a\notin\cM$. Then $\U=ag+h$, for some $g\in\cA$ and
  $h\in\cM$. Hence $\U=ag(x_0)+h(x_0)$. Since $h(x_0)\in\cM(x_0)$
  and $\cM(x_0)$ is a proper ideal, we must have $a\notin\cM(x_0)$.
  Therefore, $\cM(x_0)\subset \cM\cap A$.

  Now, assume that $x_1\neq x_0$. There is a function $g\in C(X)$ such that
  $g(x_0)=1$ and $g(x_1)=0$. Hence $g\notin\cM$ and so
  there exist $h\in\cA$ and $f\in \cM$ such that $\U=f+gh$.
  Hence, $\U\in \cM(x_1)$ which means that $\cM(x_1)=A$.
\end{proof}

\begin{cor}\label{if I(x0) neq A and I(x1) neq A then I is not maximal}
  Let $\cI$ be an ideal in $\cA$.
  If there exist two distinct points $x_0,x_1\in X$ such that
  $\cI(x_0)\neq A$ and $\cI(x_1)\neq A$,
  then $\cI$ cannot be maximal.
\end{cor}

\begin{thm}\label{vp(M)=M cap C(X)}
  If $\cM$ is a maximal ideal in $\cA$, then $\phi(\cM)$, for some
  character $\phi\in\mis(A)$, is
  a maximal ideal in $C(X)$, and $\psi(\cM)=C(X)$ for $\psi\neq \phi$.
  Moreover, $\phi(\cM)=\cM\cap C(X)$.
\end{thm}

\begin{proof}
  By Theorem \ref{M(x)=M cap A}, there exists a point
  $x_0\in X$ such that $\cM(x_0)$
  is a maximal ideal in $A$. By \cite[Theorem 11.5]{Rudin-FA},
  there is a character $\phi:A\to\C$ with
  $\cM(x_0)=\ker \phi$. We show that
  \begin{equation}\label{eqn:phi(cM)=cM cap C(X)}
    \cM \cap C(X) = \phi(\cM) = \set{f\in C(X):f(x_0)=0}.
  \end{equation}

  \noindent
  The inclusions from left to right in the above equations are obvious.
  If $f\in C(X)$ and $f\notin \cM$, then
  $\U=fg+h$, for some $g\in\cA$ and $h\in\cM$. Hence,
  $\U=f(x_0)g(x_0)$ which means that $f(x_0)\neq0$.

  Now, let $\psi\in\mis(A)$ with
  $\phi(\cM)=\psi(\cM)$. If $a\in\ker\phi$ then $a\in\cM$
  an thus $\psi(a)\in\phi(\cM)$. By \eqref{eqn:phi(cM)=cM cap C(X)},
  since $\psi(a)$ is constant,
  we get $\psi(a)=0$.
  This shows that $\ker\phi\subset \ker \psi$
  and so $\phi=\psi$.
\end{proof}

We can summarize the above results in the following statement.

\begin{thm}\label{maximal ideals in cA}
  For an ideal $\cM$ in $\cA$ the following are equivalent.
  \begin{enumerate}[\upshape(i)]
    \item \label{item:M is maximal}
    The ideal $\cM$ is maximal.

    \item \label{item:M cap A and M cap C(X) are maximal}
    The ideals $\cM\cap A$ and $\cM\cap C(X)$ are maximal,
    respectively in $A$ and $C(X)$.

    \item \label{item:M(x) is maximal for a unique x}
    There is a unique point $x_0\in X$ such that $\cM(x_0)\neq A$ and
    $\cM(x_0)$ is maximal in $A$.


    \item \label{item:M if of the form (f:f(x) in M)}
    There exist a point $x_0\in X$ and a maximal ideal $M$ in $A$
    such that
    \[
      \cM=\set{f\in \cA: f(x_0)\in M}.
    \]
  \end{enumerate}
\end{thm}

\begin{proof}
  $\eqref{item:M is maximal}\Rightarrow\eqref{item:M cap A and M cap C(X) are maximal}$:
  It follows from Theorem \ref{M(x)=M cap A} and Theorem \ref{vp(M)=M cap C(X)}.

  $\eqref{item:M cap A and M cap C(X) are maximal}\Rightarrow\eqref{item:M(x) is maximal for a unique x}$:
  Since $\cM\neq \cA$, by Lemma \ref{if I(x)=J(x) for all x then I=J}, there is a point
  $x_0\in X$ such that $\cM(x_0)\neq A$. The proper ideal $\cM(x_0)$ is maximal, since it
  contains the maximal ideal $\cM\cap A$. On the other hand, since $\cM\cap C(X)$ is
  maximal in $C(X)$, we must have
  \[
    \cM\cap C(X) = \set{f\in C(X): f(x_0)=0}.
  \]

  \noindent
  Now, if $x_1\neq x_0$, there is a function $f\in C(X)$ such that
  $f(x_0)=0$ and $f(x_1)=1$. It shows that $f\in \cM\cap C(X)$.
  Hence $\U=f(x_1)\in \cM(x_1)$ and thus $\cM(x_1)=A$.

  $\eqref{item:M(x) is maximal for a unique x} \Rightarrow \eqref{item:M if of the form (f:f(x) in M)}$:
  Set $M=\cM(x_0)$ and $\cI=\set{f\in \cA: f(x_0)\in M}$.
  Then $\cM\subset \cI$. If $x\in X\setminus\set{x_0}$ then, by
  the assumption, $A=\cM(x)\subset \cI(x)\subset A$.
  Hence we must have $\cI(x_0)\neq A$ since $\cI\neq A$.
  Since $\cM(x_0)$ is maximal and
  $\cM(x_0) \subset \cI(x_0)$, we get $\cM(x_0)=\cI(x_0)$.
  Therefore, for every $x\in X$, we have shown that $\cM(x)=\cI(x)$.
  Lemma \ref{if I(x)=J(x) for all x then I=J} shows that  $\cM=\cI$.

  The implication
  $\eqref{item:M if of the form (f:f(x) in M)}\Rightarrow\eqref{item:M is maximal}$
  is easy to prove; see \cite[Lemma 2.1]{Yood}.
\end{proof}

\end{document}